\documentclass[10pt]{article}
\usepackage{amsmath}
\usepackage{amsthm}
\usepackage{amsfonts}
\usepackage{amssymb}
\usepackage{amscd}
\usepackage[all]{xy}
\usepackage{graphicx}
\usepackage{xcolor}
\usepackage{enumerate}

\font\smallsc=cmcsc10
\font\smallsl=cmsl10

\newtheorem{theorem}{Theorem}

\newtheorem{corollary}[theorem]{Corollary}

\newtheorem{lemma}[theorem]{Lemma}

\newtheorem{proposition}[theorem]{Proposition}

\newcommand{\gon}{\mathrm{gon}}

\newcommand{\Pbb}{\mathbb P}
\newcommand{\Cbb}{\mathbb C}
\newcommand{\Ocal}{\mathcal O}

\newcommand{\mi}{m^{(i)}}
\newcommand{\n}[1]{n^{(#1)}}
\newcommand{\nlinha}[1]{{n'}^{(#1)}}

\renewcommand{\:}{\colon }

\begin{document}

\title{Characterizing gonality for two-component stable curves}
\author{Juliana Coelho\\{\scriptsize julianacoelhochaves@id.uff.br}
\and Frederico Sercio\\{\scriptsize fred.feitosa@ufjf.edu.br}}

\maketitle

\begin{abstract}
It is a well-known result that a stable curve of compact type over $\Cbb$ having two components is hyperelliptic if and only if 
both components are hyperelliptic and the point of intersection is a Weierstrass point for each of them. 
With the use of admissible covers, we generalize this characterization in two ways: 
for stable curves of higher gonality having two smooth components and one node; 
and for hyperelliptic and trigonal stable curves having two smooth non rational components and any number of nodes.
\end{abstract}


\section{Introduction}


A smooth curve $C$ is said to be $k$-gonal if it admits a degree-$k$ map to $\mathbb P^1$. 
The gonality, that is, the minimum $k$ such that $C$ is $k$-gonal, is an important numerical invariant in the study of algebraic curves. When considering smooth curves in families, one feels the need to consider singular curves as well, since smooth curves can degenerate to singular ones. 

There are a few possible ways to extend the notion  of gonality to singular curves. 
One way is to simply take the same definition as in the smooth case, as is done in \cite{sthor} and   \cite{renato}.
The drawback here is that this definition does not work very well in families, since not all curves that are limits of smooth $k$-gonal ones have maps of degree $k$ to $\Pbb^1$. 
One solution is to consider limits of maps from smooth curves to $\Pbb^1$.
This can be done using stable maps, as in \cite{cv},
or using Harris and Mumford's admissible covers \cite{HM82}. In this paper we adopt the latter point of view, as was done previously in \cite{paper1}.

It is worth mentioning other approaches to the problem.
For instance, since a map from a smooth curve to $\Pbb^1$ is given by a linear system on the curve, 
one can consider Eisenbud and Harris' limit linear series introduced \cite{eisenharris} (see for instance \cite{em} and \cite{osser}). 
Furthermore, some authors have been working on relating the gonality of stable curves to the gonality of graphs and tropical curves (see \cite{caporaso}).




\subsection{Main Results}

In this paper we always work over $\mathbb{C}$.
As we mentioned, we consider a stable curve to be $k$-gonal if it is a limit of smooth $k$-gonal curves.
It is a well-known result that a stable curve of compact type with two components is hyperelliptic if and only if 
both components are hyperelliptic and the point of intersection is a Weierstrass point for each of them. 
In this work we generalize this characterization with the use of admissible covers.

A $k$-sheeted quasi admissible cover for a nodal curve $C$ is a finite morphism $\pi\:C\rightarrow B$ of degree $k$ satisfying a few conditions, where $B$ is a nodal curve of genus $0$ (see Section \ref{sec:admiss} for the precise definition). A quasi admissible cover is said to be admissible if it is simply branched away from the singular points of $C$. In \cite{HM82}, Harris and Mumford showed that a stable curve $C$ is $k$-gonal if there exists an admissible cover for a curve $C'$ stably equivalent to $C$, that is, such that $C$ is obtained from $C'$ by contracting to a point some of the rational components of $C'$ meeting the rest of the curve in just one or two points.

In Section \ref{sec:admiss} we consider quasi admissible and admissible covers. 
We note in Theorem \ref{thm:quasitoadmiss} that a quasi admissible cover can always be extended to an admissible one.
The main concern of this section is the problem of gluing quasi admissible covers, in the sense of \cite{paper1}. In Theorem \ref{thm:gluing} we adapt a construction of \cite{paper1} to this case. 
The proof is \textit{mutatis mutandis} that of \cite[Thm 3.7]{paper1}, although we needed some extra hypotheses 
to obtain an admissible cover (see Theorem \ref{thm:gluing} for the precise statement). 

In Section \ref{sec:2comp} we restrict ourselves to two-component stable curves. 
 Proposition \ref{prop:gluing2} deals with the construction of Theorem \ref{thm:gluing}  in this case
and, as a consequence, we obtain in  Corollary \ref{cor:gonchar} a bound on the gonality of a two-component stable curve, that is, the minimum $k$ such that the curve is $k$-gonal. 
The theorem below summarizes these results:

\medskip
\noindent{\bf Theorem A.} \; {\it 
Let $C$ be a stable curve with smooth components $C_1$ and $C_2$ 
and nodes $n_1,\ldots,n_\delta$. 
For each $j=1,\ldots, \delta$ let $n_j^{(1)}\in C _1$ and $n_j^{(2)}\in C_2$ be the branches of $n_j$.

For $i=1,2$
let $\Pi_i(k_i)$ be the set of degree-$k_i$ maps $\pi_i\colon C_i\rightarrow \Pbb^1$ such that
$$\pi_i(n_1^{(i)})=\ldots=\pi_i(n_\delta^{(i)}).$$
For $\pi_i\in\Pi_i(k_i)$ let  $e_{\pi_i}(n_j^{(i)})$ be the ramification index of $\pi_i$ at $n_j^{(i)}$, for $i=1,2$ and $j=1,\ldots,\delta$.
For $\pi_1\in\Pi_1(k_1)$ and $\pi_2\in\Pi_2(k_2)$ set 
$$e(\pi_1,\pi_2)=\sum_{j=1}^\delta \min\{e_{\pi_1}(n_j^{(1)}),\, e_{\pi_2}(n_j^{(2)})\}.$$

\begin{enumerate}[(i)]
\item If $\pi_1\in\Pi_1(k_1)$ and $\pi_2\in\Pi_2(k_2)$, then
$C$ is $(k_1+k_2-e(\pi_1,\pi_2))$-gonal;

\item If both $\Pi_1(\gon(C_1))$ and $\Pi_2(\gon(C_2))$ are non empty, then 
$$\gon(C)\leq\gon(C_1)+\gon(C_2)-\max\{e(\pi_1,\pi_2)\ | \ \pi_i\in\Pi_i(\gon(C_i)),\; i=1,2 \}.$$ 
\end{enumerate}
 }\medskip

The inequality above is an equality when $C$ has only one node. 
To see that, we show  in Proposition \ref{prop:converse} that, in some sense, the construction in Theorem \ref{thm:gluing} is optimal (see Proposition \ref{prop:converse} for the precise statement). 
This allows us to characterize in Theorem \ref{thm:goncharcpt} the gonality of a stable curve with 
two smooth  components and a single node:

\medskip
\noindent{\bf Theorem B.} \; {\it 
Let $C$ be a stable curve with smooth  components $C_1$ and $C_2$ and a single node $n$.
For $i=1,2$, let  $n^{(i)}\in C_i$ be the branch of $n$ in $C_i$ and 
let $e_i$ be the maximum ramification index of  a degree-$\gon(C_i)$ map $C_i\rightarrow \mathbb{P}^1$ at $n^{(i)}$.
Then $\gon(C)=\gon(C_1)+\gon(C_2)-e$, where
$e=\min\{e_1,e_2\}$.
 }\medskip

In Section \ref{sec:hyptrig} we use results of the previous sections to characterize hyperelliptic and trigonal curves having two smooth non rational components. 
For hyperelliptic curves we show in Theorem \ref{thm:hyp} that, besides the 
well-known case of stable curves of compact type, there is only one other possibility:

\medskip
\noindent{\bf Theorem C.} \; {\it 
Let $C$ be a stable curve  with two smooth non rational components. 
Then $C$ is hyperelliptic if and only if both components are hyperelliptic and one of the following cases hold:
\begin{enumerate}[(i)]
\item $C$ has one node and the branches of the node in each component have ramification index 2;

\item $C$ has two nodes and the branches of the nodes in each component are
points of ramification index 1 having the same image under a degree-2 map.
\end{enumerate}
 }\medskip

Finally, in Theorem \ref{thm:trig} we obtain a characterization of trigonal stable curves with two smooth non rational components. As in the hyperelliptic case, such a curve can have at most three nodes, and its components are either hyperelliptic or trigonal. 
However, in sharp contrast with the hyperelliptic case, there are eleven different  possibilities depending on
the number of nodes, 
the gonality of the components
and the behaviour of the branches of the nodes (see Theorem \ref{thm:trig} for the precise statement).

\section{Technical background}\label{sec:back}

In this paper we always work over the field of complex numbers $\mathbb{C}$.

A \textit{curve} $C$ is a connected, projective and reduced scheme of dimension $1$ over $\Cbb$. 
We denote by $C^{\text{sing}}$ the singular locus of $C$, and by $C^{\text{sm}}$ the smooth locus of $C$.
The \emph{genus} of $C$ is $g(C):=h^1(C,\Ocal_C)$.
A \textit{subcurve} $Y $ of $C$ is a reduced subscheme of pure dimension $1$, or equivalently, a reduced union of irreducible components of $C$. 
If $Y\subseteq C$ is a subcurve, we say $Y^{c}:=\overline{C\smallsetminus Y}$ is the \emph{complement} of $Y$ in $C$.

A \textit{nodal curve} $C$ is a curve with at most ordinary double points, called \emph{nodes}. 
A node is said to be \emph{separating} if there is a subcurve $Y$ of $C$ such that $Y\cap Y^c=\{n\}$. 
The subcurves $Y$ and $Y^c$ are then said to be \emph{tails} of $C$ associated to the separating node $n$.
A \emph{rational tail} is a tail of genus 0.
More generally, a \emph{rational chain}  is a nodal curve of genus 0. 
We remark that a curve  is a rational chain if and only if its nodes are all separating and its components are all rational.

Let  $g$ and $n$ be non negative integers such that $2g-2+n>0$. A \textit{$n$-pointed stable curve of genus $g$} is a curve $C$ of genus $g$ together with $n$ distinct \emph{marked points} $p_1,\ldots,p_n\in C$  such that for every  smooth rational component $E$ of $C$,  
 the number of points in the intersection   $E\cap E^{c}$ plus the number of indices $i$ such that  $p_i$ lies on $E$ is at least three. 
A \textit{stable curve} is a 0-pointed stable curve.

Let $C$ be a nodal curve e let $\nu\:\widetilde{C}\rightarrow C$ be its
normalization. Let $n$ be a node of $C$. 
The points $\n{1} ,\n{2} \in
\widetilde{C}$ such that $\nu(\n1)=\nu(\n2)=n$ are said to be the \textit{branches over the node} $n$.

Let $\pi\:C\rightarrow B$ be a finite map between curves and let $p\in C$ and $q\in B$ be smooth points such that $\pi(p)=q$. Let $x$ be a  local parameter of $C$ around $p$ and $t$ be a local parameter of $B$ around $q$. 
Then, locally around $p$, the map $\pi$ is given by $t=x^e$, where $e$ is the \emph{ramification index} of $\pi$ at $p$, denoted by $e_\pi(p)$.
Note that $1\leq e_\pi(p)\leq k$, where $k$ is the degree of $\pi$. 
 We say \emph{$\pi$ is ramified at $p$}, or that \emph{$p$ is a ramification point of $\pi$}, if $e_\pi(p)\geq2$. 
We say $\pi$ is \emph{totally ramified} at $p$ if $e_\pi(p)=k$.

Let $\pi\:C\rightarrow B$ be a finite map of degree $k$ between curves and let $q$ be a  smooth point $B$. Then
\begin{equation}\label{eq:degk}
\sum_{p\in\pi^{-1}(q)}e_{\pi}(p)=k.
\end{equation}
The \emph{multiplicity of $q$ in $B$ with respect to $\pi$} is defined as
$$d_{\pi}(q)=\sum_{p\in\pi^{-1}(q)}(e_{\pi}(p)-1).$$
Note that $0\leq d_{\pi}(q)\leq k-1$.
Moreover, $d_{\pi}(q)= k-1$ if and only if there exists a single point $p\in C$ such that $\pi(p)=q$, and in this case $\pi$ is totally ramified at $p$. The \emph{branch points of $\pi$} are the smooth points $q$ in $B$ such that $d_\pi(q)\geq 1$.
The Riemann-Hurwitz formula counts the number $b(\pi)$ of  branch points of $\pi$, with multiplicity. In 
 particular, if $B$ is rational, we have
\begin{equation}\label{eq:RH}
b(\pi)=\sum_{q\in B^{\text{sm}}}d_{\pi}(q)
=2g(C)+2k-2.
\end{equation}



\section{Quasi admissible and admissible covers}\label{sec:admiss}

A smooth curve is \emph{$k$-gonal} if it admits a $\mathfrak{g}_{k}^{1}$, that is, a line bundle of degree $k$ having at least two sections.
Equivalently, a smooth curve is $k$-gonal if it admits a map of degree  $k$ or less  to $\mathbb{P}^{1}$. Now, a stable curve $C$ is 
 \emph{$k$-gonal} if it is a limit of smooth $k$-gonal curves in $\overline M_g$. 
 More precisely, 
a \textit{smoothing} of a curve $C$ is a proper and flat morphism 
$f\:\mathcal{C}\rightarrow \mathrm{Spec}(\mathbb{C}[[t]])$ 
whose fibers are curves, with special fiber $C$ and
such that $\mathcal{C}$ is regular.
A stable curve $C$ is then said to be $k$-gonal if it admits a
 smoothing $f\:\mathcal{C}\rightarrow S$ whose general fiber is
a $k$-gonal smooth curve and the special fiber is $C $.

Alternatively, $k$-gonal stable curves can be characterized in terms of admissible covers.
A \emph{$k$-sheeted quasi admissible cover} consists of  a finite morphism $\pi\:C\rightarrow B$ of degree $k$, such that $C$ and $B$ are nodal curves, with $g(B)=0$,   and
\begin{enumerate}[(1)]
\item $\pi^{-1}\left(B^{\text{sing}}\right)  =C^{\text{sing}}$;
\item for every  subcurve $Z\subset B$ we have
$$|Z\cap Z^c| + \sum_{q\in Z\cap B^{\text{sm}}}d_{\pi}(q)\geq 3;$$
\item for every node $q$ of $B$ and every node $n$ of $C$ 
lying over it, the
two branches of $C$ over $n$ map to the branches of $B$ over $q$ with the same
ramification index.
\end{enumerate}

Let $\pi\colon C\rightarrow B$ be a $k$-sheeted quasi admissible cover. Let   $n$ be a node of $C$ and set $q=\pi(n)$. By condition (1), $q$ is a node of $B$. Then condition (3) above implies that, locally around $n$, the curve $C$ can be described as $xy=t$  and, locally around $q$, the curve  $B$ can be described as $uv=t^e$ for some $e$. Moreover, the map $\pi$ is given by $u=x^e$ and $v=y^e$ and $e$ is the ramification index of $\pi$ at $n$.

An  \emph{admissible cover} is a quasi admissible cover $\pi$ satisfying
\begin{enumerate}
\item[(4)] $\pi$ is simply branched away from $C^{\text{sing}}$, that is, 
over each smooth point of $B$ there exists at most one point of $C$ 
where $\pi$ is ramified and this point has ramification index 2. 
\end{enumerate}
Condition (4) is equivalent to saying that $d_\pi(q)\leq 1$ for every smooth point $q\in B$.
In particular, if condition (4) holds, then condition (2) means that  $B$ is a stable pointed curve of genus $0$, when considered with the branch points of $\pi$.

If $\pi\:C\rightarrow B$ is a finite map, we say a subcurve $Z$ of $B$  \emph{violates condition (2) for $\pi$} if the inequality in condition (2) is not satisfied. The next lemma shows it is enough to check condition (2) on the irreducible components of $B$.

\begin{lemma} \label{lem:condition2} 
Let $\pi\:C\rightarrow B$  be a finite map between curves such that $g(B)=0$.  
If a subcurve $Z$ of $B$ violates condition (2) for $\pi$, then there is an irreducible component of $Z$ that also  violates condition (2) for $\pi$. 
\end{lemma}
\begin{proof}
Let  $Z$ be a subcurve of $B$ that violates condition (2) for $\pi$.
If $Z$ is irreducible, then there is nothing to prove. Now consider $Z$ to be reducible. 

If $Z=B$ then $Z\cap Z^c$ is empty and hence $\pi$ has at most two branch points, counted with multiplicity. But then any irreducible rational tail $W$ of $B$ not containing both of these points will violate condition (2) for $\pi$. 

Now  assume $Z$ is a proper subcurve of $B$. Since $Z$ violates condition (2) for $\pi$, then $Z$  contains at most one branch point of $\pi$  and $Z\cap Z^C$ consists of one or two points. Hence we may choose $W$ to be any component of $Z$ not containing the branch point and such that $W\cap W^c$  consists of one or two points.
\end{proof}

We remark that not all finite maps from a smooth curve to $\Pbb^1$ are quasi admissible.

\begin{lemma} \label{lem:ratldeg2} 
Let $\pi\:C\rightarrow B$  be a finite map of degree $k$ between curves, such that $g(B)=0$. If $g(C)=0$ and $k\leq 2$ then $B$ violates condition (2). 
Moreover, if $C$ and $B$ are smooth then $\pi$ is not quasi admissible if and only if $g(C)=0$ and $k\leq 2$.
\end{lemma}
\begin{proof}
By (\ref{eq:RH}), the map $\pi$ has $b(\pi)\leq2$ ramification points counted with multiplicity if and only if 
$g(C)=0$ and $k\leq 2$
or $g(C)=1$ and $k=1$. 
Note however that the later situation does not happen since $g(B)=0$.
Moreover, if $C$ is smooth, then then conditions (1) and (3) always hold.
\end{proof}

Let $C$ be a nodal curve. We say that a nodal curve $C^{\prime}$ is \emph{stably equivalent} to $C$ if $C$ can be obtained from $C^{\prime}$ by contracting to a point some of the smooth rational components of $C^{\prime}$ meeting the other
components of $C^{\prime}$ in only one or two points. In that case, $g(C)=g(C')$ and there is a contraction map $\tau\:C'\rightarrow C$.
We say that a point $p'\in C'$ \emph{lies over} a point $p\in C$ if $\tau(p')=p$.

The following result is a consequence of \cite[Thm. 4, p. 58]{HM82} and relates the notion of admissible covers to that of gonality of stable curves.

\begin{theorem}[Harris-Mumford]\label{HarrMum}
A stable curve $C$ is $k$-gonal if and only if
there exists a $k$-sheeted admissible cover $C^{\prime}\rightarrow B$, 
where $C^{\prime}$ is stably equivalent to $C$.
\end{theorem}
\begin{proof}
Cf. \cite[Thm. 3.160, p. 185]{HM}.
\end{proof}

We remark that condition (4) in the definition of and admissible cover is in fact not necessary, as a quasi admissible cover can always be turned into an admissible cover, as shown the following result from  \cite{paper1}. We however point out that this proccess exchanges a ramification point of index $e$ (more generally, a branch point of multiplicity $d$) of the original quasi admissible cover for $e-1$ distinct ramification points of index $2$ (more generally, $d$ branch points of multiplicity 1) of the new admissible cover.

\begin{theorem}\label{thm:quasitoadmiss}
Let $\pi\: C \rightarrow B$ be a $k$-sheeted quasi admissible cover.
Then there is a $k$-sheeted admissible cover
$\pi' \: C' \rightarrow B'$ such that $C'$ is stably equivalent to $C$ and contains $C$ as a subcurve, $B'$ is stably equivalent to $B$ and contains $B$ as a subcurve, and $\pi'|_{C}=\pi$.
\end{theorem}
\begin{proof}
Cf. \cite[Thm 3.3(a)]{paper1}.
\end{proof}


The following theorem describes a procedure to glue quasi admissible covers in a similar way to what was done in \cite[Proposition 3.7]{paper1}. The main difference is that, since here we allow the covers to be quasi admissible, 
obtaining condition (2) for the glued cover will take a few extra hypotheses.

\begin{theorem}\label{thm:gluing}
Let $C$ be a nodal curve and let $Y_1,\ldots,Y_r\subset C$ be connected subcurves such that $Y_1\cup\ldots\cup Y_r=C$ and $Y_i\cap Y_{i'}$ is finite, possibly empty, for $1\leq i\neq i'\leq r$. For $i=1,\ldots,r$  let: 
\begin{itemize}
\item $\pi_i\:Y_i'\rightarrow B_i$ be a quasi admissible cover of degree $k_i$, where $Y_i'$ is stably equivalent to $Y_i$;
\item $n_{i,1},\ldots,n_{i,\delta_i}\in C$ be the intersection points between $Y_i$ and $Y_i^c$;
\item $n_{j}^{(i)}\in Y_i$ be the branch over $n_{i,j}$, and let $\nlinha{i}_j\in Y_i'$  be a smooth point lying over $n_j^{(i)}$, for $j=1,\ldots,\delta_i$.
\end{itemize}
Set $S=\{n_{i,j}\ |\ i=1,\ldots,r,\; j=1,\ldots,\delta_i\}$ and,
for each $n\in S$, set
$$e_n=\min\{ e_{\pi_{i_0}}(\nlinha{i_0}_{j_0}), e_{\pi_{i_1}}(\nlinha{i_1}_{j_1})\ | \ n=n_{i_0,j_0}=n_{i_1,j_1}\}.$$
If for every $i=1,\ldots,r$ we have $\pi_i(\nlinha{i}_1)=\ldots=\pi_i(\nlinha{i}_{\delta_i})$ 
then there exists a finite map $\pi\: C'\rightarrow B$ of degree 
$$k_1+\ldots+k_r-\sum_{n\in S}e_n$$ satisfying conditions (1) and (3) of a quasi admissible cover, 
where $C'$ is stably equivalent to $C$, contains $Y_i'$ as a subcurve, $B$ contains $B_i$  as a subcurve, and $\pi|_{Y_i'}=\pi_i$.

Moreover if, for each $i=1,\ldots,r$, we have $g(Y_i)\neq 0$ or $k_i\geq 3$, and 
if the  component of $B_i$ containing $q_i$ meets its complement in at least 2 points or if it
contains at least two smooth branch points of $\pi_i$ counted with multiplicity,
then $\pi$ is a quasi admissible cover.
 In this case,  $C$ is $(k_1+\ldots+k_r-\sum_{n\in S}e_n)$-gonal.
\end{theorem}
\begin{proof}
The proof is similar to that of \cite[Thm. 3.7]{paper1}. 
We proceed as follows.

For $i=1,\ldots,r$ let $q_{i}:=\pi_{i}({n_1'}^{\hspace{-.02in}(i)})=\ldots=\pi_{i}({n_{\delta_i}'}^{\hspace{-.04in}(i)})$. 
For simplicity, for $j=1,\ldots,\delta_i$ we set $e_j^i = e_{\pi_{i}}(\nlinha{i}_{j})$ to be
the ramification index of $\pi_i$ at  ${n_j'}^{(i)}$. Then
\[
(\pi_{i})^{-1}(q_{i})=
e_{1}^{i}{n_1'}^{(i)} +\ldots+e_{\delta_i}^{i}{n_{\delta_i}'}^{\hspace{-.04in}(i)}+
\lambda_{1}^{i}\mi_1+\ldots+\lambda_{u_i}^{i}\mi_{u_i},
\]
where $\mi_1,\ldots,\mi_{u_i}\in Y_{i}'$. We glue (see Figure \ref{fig:gluing1}):
\begin{itemize}
\item  a copy of $\mathbb{P}^{1}$, denoted by $B'$, passing through
$B_{1},\ldots,B_r$ at $q_{1},\ldots,q_r$, respectively, and thus linking the curves.  Denote by $B$ the genus-$0$ curve thus obtained;

\item whenever $n_{i_0,j_0}=n_{i_1,j_1}$, a copy of $\mathbb{P}^{1}$ passing through $Y_{i_0}'$ and $Y_{i_1}'$ at 
${n_{j_0}'}^{\hspace{-.04in}(i_0)}$ and ${n_{j_1}'}^{\hspace{-.04in}(i_1)}$,  and thus linking both curves,
mapping to $B'$ via a map of degree 
$\max\{ e_{j_0}^{i_0}, e_{j_1}^{i_1}\}$
 ramified to order $e_{j_0}^{i_0}$ at ${n_{j_0}'}^{\hspace{-.04in}(i_0)}$ 
and to order $e_{j_1}^{i_1}$ at ${n_{j_1}'}^{\hspace{-.04in}(i_1)}$, unramified over $q_t$ for every $t\neq i_0,i_1$ and simply ramified elsewhere, where  $1\leq i_0,i_1,t\leq r$, $1\leq j_s\leq \delta_{i_s}$ and $s=0,1$ (call this copy $L_{j_0,j_1}^{i_0,i_1}$);

\item 
 a copy of $B_t$, mapping to  $B_t$ isomorphically,
at each point $m$ of $L_{j_0,j_1}^{i_0,i_1}$, distinct from ${n_{j_0}'}^{\hspace{-.04in}(i_0)}$ and ${n_{j_1}'}^{\hspace{-.04in}(i_1)}$,  
lying over $q_t$  by the map $L_{j_0,j_1}^{i_0,i_1}\rightarrow B'$, 
where  $1\leq i_0,i_1,t\leq r$ and  $1\leq j_s\leq \delta_{i_s}$, for $s=0,1$;

\item a copy of $\mathbb{P}^1$ at $\mi_{j}$, mapping to $B'$ via a degree-$\lambda_{j}^{i}$ map totally ramified at $\mi_j$,
unramified over $q_t$ for every $t\neq i$, 
and simply ramified elsewhere, 
for each $1\leq i,t\leq r$ and $1\leq j\leq u_i$ (call this copy $L_j^i$);

\item a copy of $B_t$ 
at each point $m$ of $L_j^i$ lying over $q_t$ by the map $L_{j}^i\rightarrow B'$, mapping to  $B_t$ isomorphically,
for each $1\leq i,t\leq r$ with $t\neq i$, and $1\leq j\leq u_i$.
\end{itemize}

\begin{figure}[h!]
\centering
\includegraphics[height=2.6in]{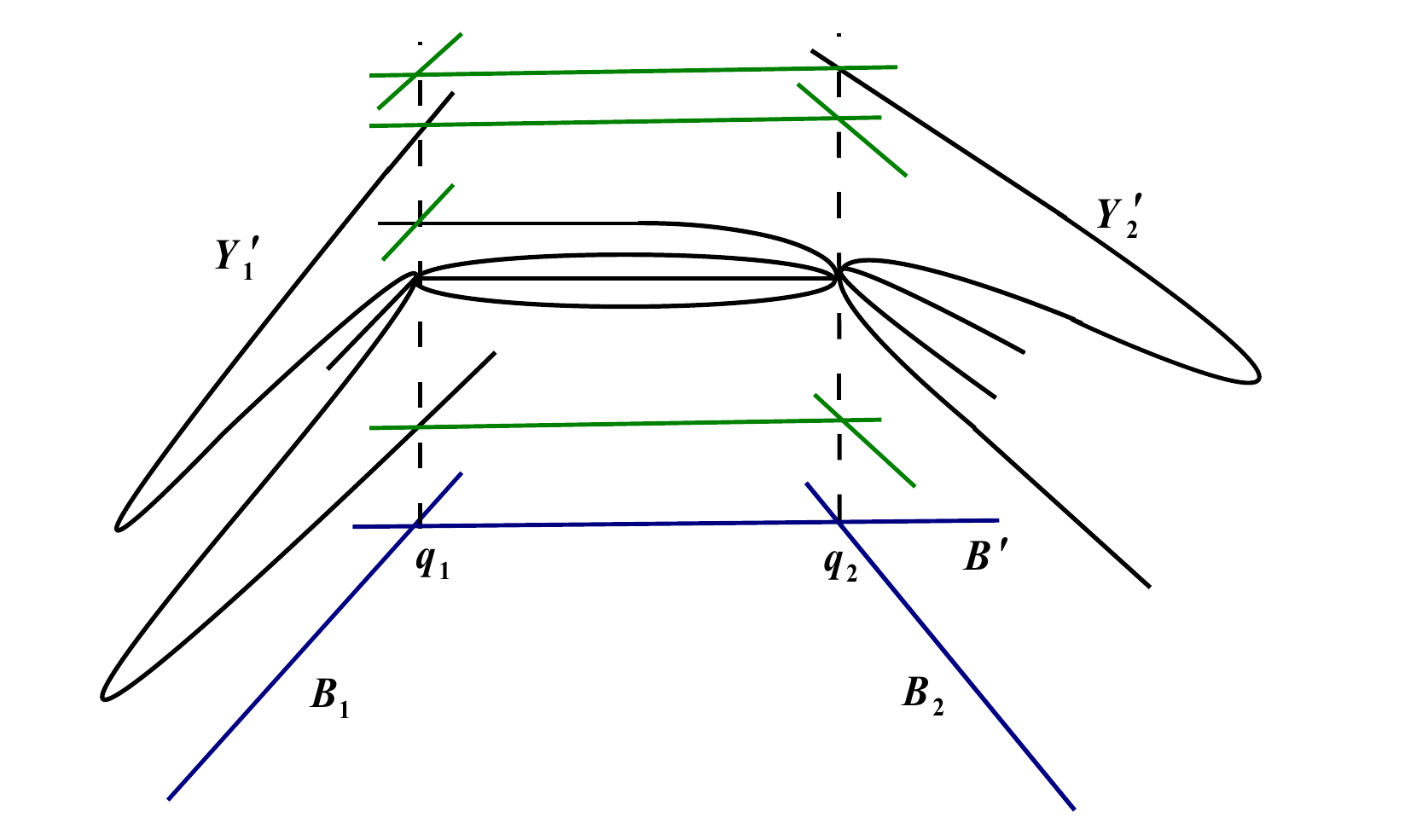}
\caption{Theorem \ref{thm:gluing} for $r=2$, $k_1=5$, $k_2=5$, $e_1^1=3$ and $e_2^1=4$}
\label{fig:gluing1}
\end{figure}

We thus obtain a nodal curve $C^{\prime}$ stably equivalent to $C$
 and a map $\pi\:C^{\prime}\rightarrow B$ given by $\pi_{i} $ when
restricted to $Y_{i}'$, and by the maps described above when restricted to the
added rational components of $C^{\prime}$. 
By construction, $\pi$ has degree $k=k_{1}+\ldots+k_{r}-\sum_{n\in S}e_n$ and satisfies conditions (1) and (3) of a quasi admissible cover.

Now we consider condition (2). Note that since $\pi_i$ is quasi admissible and $Y_i'$ is stably equivalent to $Y_i$, then $B_i$ has $b_i=b(\pi_i)=2g(Y_i)+2k_i-2$ branch points counted with multiplicity. Now, $q_i$ has multiplicity 
$d_i=d_{\pi_i}(q_i)\leq k_i-1$ 
in $B_i$ and  there are 
$b_i-d_i$ marked points in $B_i$ distinct from $q_i$, counted with multiplicity. 
Thus $B_i$ would violate condition (2) for $\pi$ if and only if $b_i-d_i\leq 1$.
Now, 
$$b_i-d_i\geq b_i-k_i+1=2g(Y_i)+k_i-1.$$ 
Since, by hypothesis, we have $g(Y_i)\geq 1$ or $k_i\geq 3$, then $b_i-d_i\geq 2$, showing that $B_i$ does not violate condition (2) for $\pi$.

Note moreover that if $Z$ is a component of $B_i$, then $Z$ also does not violate condition (2) for $\pi$. 
Indeed, if $Z$ does not contain $q_i$, then the inequality in condition (2) is equal for both $\pi$ and $\pi_i$. 
Moreover, if $Z$ does contain $q_i$ 
then, by hypothesis,
$Z$ meets its complement in at least 2 points or 
contains at least two smooth branch points of $\pi_i$ counted with multiplicity,
and again we are done.


Now we need only to check that $B'$ does not violate condition (2) for $\pi$. This is clear if $r\geq 3$. Now assume $r=2$, then $B'$ violates condition (2) for $\pi$ if and only if there are no marked points in $B'$ other than the nodes $q_1$ and $q_2$. This happens when all the maps to $B'$ have no ramification points other than the nodes $\nlinha{i}_j$. 
In the case of the maps $L_j^i\rightarrow B'$ defined in the fourth step above, this is the case  if and only if the map is an isomorphism.
In the case of maps $L_{j_0,j_1}^{i_0,i_1}\rightarrow B'$ defined in the second step above, this is the case if and only if we have
$e_{j_0}^{i_0} = e_{j_1}^{i_1}$, that is, the  ramification index of $\pi_{i_0}$ at   ${n_{j_0}'}^{\hspace{-.04in}(i_0)}$ is the same as the ramification index of $\pi_{i_1}$ at  ${n_{j_1}'}^{\hspace{-.04in}(i_1)}$. 
Hence, if  $B'$  violates condition (2) for $\pi$, then contracting $B'$ and the rational components of $C'$ mapping to it, we get a finite map $\pi'\:C''\rightarrow B''$ that satisfies conditions (1) and (2) and, since the ramification indexes agree on the branches of the nodes, we also have condition (3), see Figure (\ref{fig:gluing2}). 

\begin{figure}[h!]
\centering
\includegraphics[height=2.6in]{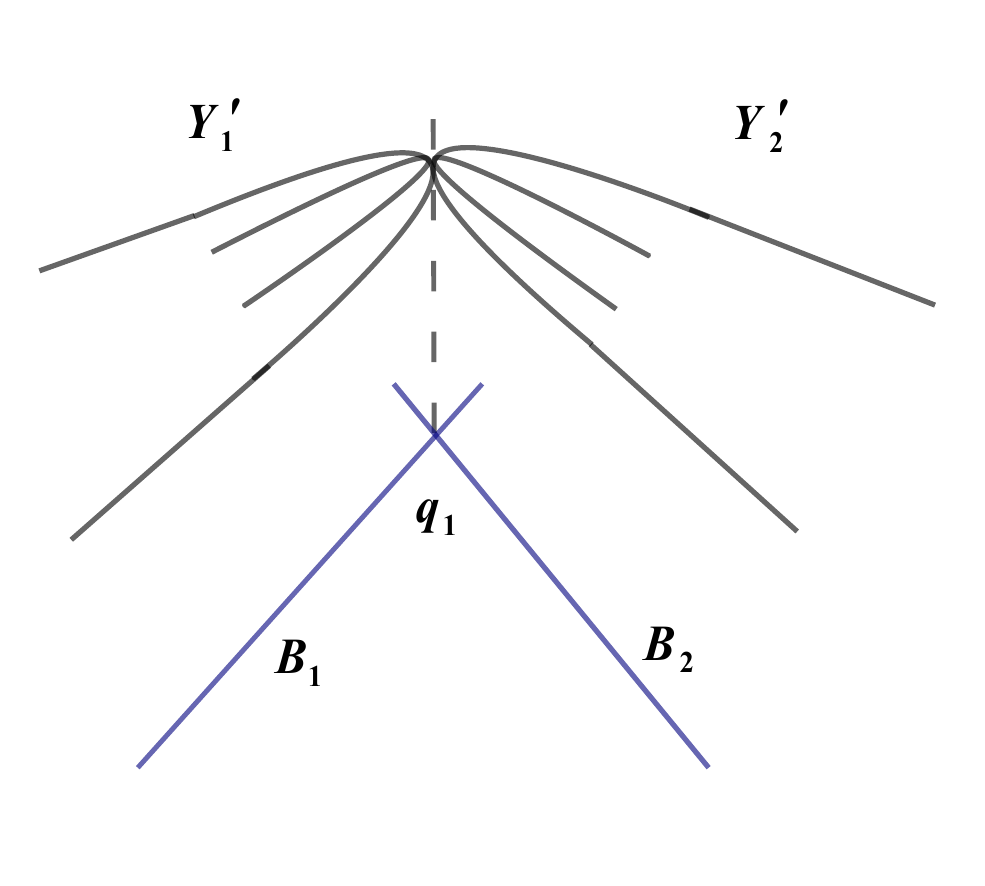}
\caption{Theorem \ref{thm:gluing} for $r=2$, $k_1=k_2=3$, $e_1^1=e_2^1=3$}
\label{fig:gluing2}
\end{figure}

Finally, the last assertion follows from  Theorem \ref{thm:quasitoadmiss}.
\end{proof}

We point out that  in order to glue two quasi admissible covers $\pi$ and $\tilde \pi$, 
we could first consider the admissible covers $\pi'$ and $\tilde\pi'$ given by Theorem \ref{thm:quasitoadmiss} and then glue $\pi'$ to $\tilde\pi'$ using \cite[Theorem 3.7]{paper1}. 
However, the resulting admissible cover will usually have degree bigger than the one obtained if we glue $\pi$ and $\pi'$ directly with the procedure of Theorem \ref{thm:gluing}. 
In fact, the degree will be the same in both constructions if and only if the minimum index $e_n$ is equal to $1$ for every node $n$.
We will see on Proposition \ref{prop:converse} that, in the case of two-component stable curves, the construction in Theorem \ref{thm:gluing} is optimal, that is, the degree of the map is the minimal possible degree.

For the sake of completeness, we remark that the constructions of \cite[Theorem 3.4]{paper1} gluing two points of the same cover can also be carried out 
for a quasi admissible cover $\pi$, with the appropriate changes, as in the previous proposition. However, in this case there is no change in the degree of the cover obtained if we first use Theorem \ref{thm:quasitoadmiss} to obtain an admissible cover $\pi'$, and then use  \cite[Theorem 3.4]{paper1} to glue points of $\pi'$.

\section{Two-component stable curves}\label{sec:2comp}

We define the \emph{gonality} of a stable curve $C$, denoted by $\gon(C)$, as the smallest $k$ such that $C$ is $k$-gonal.  It is a difficult problem to try to obtain the gonality of a stable curve from the gonality of its components. We do however have upper bounds. 

\begin{proposition}\label{prop:gonpp1}
Let $C$ be a stable curve.
\begin{enumerate}[(i)]
\item  Let $Y_1$ and $Y_2$ be stable subcurves of $C$ such that $Y_1\cup Y_2=C$ and $Y_1\cap Y_2$ is a finite set of $\delta$ nodes of $C$. Then 
$$\gon(C)\leq \gon(Y_1)+\gon(Y_2)+\delta-2.$$

\item Let $C_1,\ldots, C_p$ be the components of $C$ and let $\delta$ be the number of nodes of $C$ in the intersection of any two components of $C$. Then 
$$\gon(C)\leq \gon(C_1)+\ldots+\gon(C_p)+\delta-2(p-1).$$
\end{enumerate}
\end{proposition}
\begin{proof}
Follows from \cite[Theorem 3.9 and Corollary 3.10]{paper1}.
\end{proof}

The bounds in Proposition \ref{prop:gonpp1} are obtained from the admissible covers constructed in \cite{paper1}.  As we pointed out in the end of the previous section, the construction on Theorem \ref{thm:gluing} can produce better bounds. We will restrict ourselves to two-component stable curves.

\begin{proposition}\label{prop:gluing2}
Let $C$ be a stable curve with smooth components $C_1$ and $C_2$ 
and nodes $n_1,\ldots,n_\delta$. 
For each $j=1,\ldots, \delta$ let $n_j^{(1)}\in C _1$ and $n_j^{(2)}\in C_2$ be the branches of $n_j$.
Assume that, for each $i=1,2$, there exists a finite map $\pi_i\:C_i\rightarrow \mathbb{P}^1$ of degree $k_i$ 
such that 
$\pi_i(n_1^{(i)})=\ldots=\pi_i(n_\delta^{(i)})$. 
Then $C$ is $(k_1+k_2-(e_1+\ldots+e_\delta))$-gonal, where
$e_j=\min\{e_{\pi_1}(n_j^{(1)}),e_{\pi_2}(n_j^{(2)})\}$, for $j=1,\ldots,\delta$
\end{proposition}
\begin{proof}
Let $\pi\:C'\rightarrow B$ be the finite map constructed in Theorem \ref{thm:gluing}.
If for each $i=1,2$ we have either $g(C_i)\neq0$ or $k_i\geq3$, then the result follows directly from Theorem \ref{thm:gluing}, since by Lemma \ref{lem:ratldeg2}, in this case the maps $\pi_1$ and $\pi_2$ are quasi admissible, and satisfy the conditions of Theorem \ref{thm:gluing}.

Now assume $g(C_i)=0$ and $k_i\leq 2$ for some $i=1,2$. Then, since the branches $n_j^{(i)}$ are on the same fiber of $\pi_i$ for $j=1,\ldots,\delta$, we must have $\delta=1$ or  $\delta=2$. Now, if $\delta=1$ then $C$ is not stable. So we must have $\delta=2$ and $k_i=2$. Since $\pi_i$ is of degree 2, then $\pi_i$ is unramified at the points $n_1^{(i)}$ and $n_2^{(i)}$. Setting $B_i$ to be the image of $\pi_i$, we see that the ramification points of $\pi_i$ give two marked points on  $B_i$, and hence  $B_i$ does not violate condition (2) for $\pi$. Thus $\pi$ is quasi admissible and, by Theorem \ref{thm:quasitoadmiss}, the result is proven.
\end{proof}
 
In terms of gonality, this gives the following upper bound.

\begin{corollary}\label{cor:gonchar}
Let $C$ be a stable curve with smooth  components $C_1$ and $C_2$ 
and nodes $n_1,\ldots,n_\delta$. 
For each $j=1,\ldots, \delta$ let $n_j^{(1)}\in C _1$ and $n_j^{(2)}\in C_2$ be the branches of $n_j$. 
Let $\Pi_i$ be the set of 
degree-$\gon(C_i)$ maps $\pi_i\colon C_i\rightarrow \Pbb^1$ such that
$\pi_i(n_1^{(i)})=\ldots=\pi_i(n_\delta^{(i)})$, for $i=1,2$. 
For $\pi_1\in\Pi_1$ and $\pi_2\in\Pi_2$ let 
$$e(\pi_1,\pi_2)=\sum_{j=1}^\delta \min\{e_{\pi_1}(n_j^{(1)}),\, e_{\pi_2}(n_j^{(2)})\}.$$
If both $\Pi_1$ and $\Pi_2$ are non empty, then 
$$\gon(C)\leq\gon(C_1)+\gon(C_2)-\max\{e(\pi_1,\pi_2)\ | \ \pi_1\in\Pi_1,\; \pi_2\in\Pi_2  \},$$ 
\end{corollary}
\begin{proof}
Assume $\Pi_i$ is non empty and let $\pi_i\in \Pi_i$, for $i=1,2$ such that 
$$e(\pi_1,\pi_2)= \max\{e(\pi_1,\pi_2)\ | \ \pi_1\in\Pi_1,\; \pi_2\in\Pi_2  \}.$$ 
Then by Proposition \ref{prop:gluing2} there exists an admissible cover $\pi\: C'\rightarrow B$ for some $C'$ stably equivalent to $C$ such that
$$\gon(C)\leq \deg(\pi)=\gon(C_1)+\gon(C_2)-e(\pi_1,\pi_2).$$
\end{proof}

We will show in Theorem \ref{thm:goncharcpt} that, if $C$ is a two-component curve with one node, then the inequality obtained in the previous corollary is actually an equality. 
For this we first show in Proposition \ref{prop:converse} that, in some sense,  the construction in Theorem \ref{thm:gluing} is optimal.

\begin{lemma}\label{lem:Tj}
Let $C$ be a nodal curve with smooth components $C_1$ and $C_2$ and nodes $n_1,\ldots,n_\delta$. 
Let  $\pi\: C'\rightarrow B$ be a quasi admissible cover of degree $k$, where $C'$ is stably equivalent to $C$ with contraction map $\tau\: C'\rightarrow C$.

For $i=1,2$,  
$C_i$ is an irreducible component of $C'$ and we set  
$B_i=\pi(C_i)$,  
$\pi_i=\pi|_{C_i}\: C_i\rightarrow B_i$,
and 
$$Q_i=\{\pi_i(n_j^{(i)})\ |\ j=1,\ldots,\delta\},$$
where $n_j^{(i)}\in C_i$ is the branch of $n_j$ in $C_i$,
for each $j=1,\ldots, \delta$.
 
If $B_1\neq B_2$ then:
 
 \begin{enumerate}[(i)]
 \item There exist unique $q_1\in Q_1$ and $q_2\in Q_2$ such that  
either $q_1=q_2$ or  
$B_0=B_1^c\cap B_2^c$ is a rational chain containing $q_1$ and $q_2$, but not containing $B_1$ or $B_2$.

 \item Let $m\in C_i$ such that $\pi(m)=q_i$ and let $T=\tau^{-1}(\tau(m))$, for some $i=1,2$.
Then $\pi(T)$ contains $q_1$ and $q_2$ and, if $q_1\neq q_2$, then $\pi(T)$ contains $B_0$.
Moreover, for each component $B'$ of $B$, either $T\cap \pi^{-1}(B')$ is finite, possibly empty, or  
the restriction of $\pi$ to $T\cap \pi^{-1}(B')$ has degree at least $e_{\pi_i}(m)$.
 \end{enumerate}
\end{lemma}
\begin{proof}
Since $C_i$ is smooth, it follows from \cite[Lemma 3.2]{paper1} that $C_i$ is a component of $C'$, for $i=1,2$. 

Let's show (i). If $Q_1\cap Q_2\neq \emptyset$, we let $q_1=q_2$ be a point in this intersection.
Note that since $B_1\neq B_2$ and $g(B)=0$, there is exactly one such point.
Now assume $Q_1\cap Q_2=\emptyset$. 
Let  $B_0$ be the intersection between the complements of $B_1$ and $B_2$ in $B$, that is
$$B_0=B_1^c\cap B_2^c.$$
Since $B$ is a connected nodal curve of genus 0, then $B$ is tree-like and since $B_1$ and $B_2$ are components of $B$, then $B_0$ is a rational chain in $B$. 
Since $B$ is tree-like and both $B_0$ and $B_i$ are connected, the intersection $B_0\cap B_i$ consists of a single point and we let $q_i$ be the point in $B_0\cap B_i$, for $i=1,2$. 
Note that $B_0$ is 
the intersection of 
the tail  of $B$ associated to $q_1$ not containing $B_1$ with 
the  tail  of $B$ associated to $q_2$ not containing $B_2$.


Now we show (ii). Let $m\in C_i$ such that $\pi(m)=q_i$ and let $T=\tau^{-1}(\tau(m))$. 
Note that $T$ is either a point or a rational chain.
If $T$ is a point, then $\tau(m)$ is a node of $C$, that is, $\tau(m)=n_j$ for some $j\in\{1,\ldots,\delta\}$. But then 
$n_j^{(1)},n_j^{(1)}\in T$ and hence $\pi(n_j^{(1)}),\pi(n_j^{(2)})\in \pi(T)$. Since $\pi(T)$ is a point and $\pi(n_j^{(i)})\in Q_i$ for $i=1,2$, we have $Q_1\cap Q_2\neq \emptyset$ and $\pi(n_j^{(i)})=q_i$ for $i=1,2$. In particular, $q_1,q_2\in\pi(T)$.

Now assume $T$ is a rational chain.
We have two cases to consider. 
If $m=n_j^{(i)}$ for some $j\in\{1,\ldots,\delta\}$ then $T$ contains $n_j^{(1)}$ and $n_j^{(2)}$, and hence $\pi(T)$ contains $q_1$ and $q_2$. In particular, $q_1\neq q_2$ and $\pi(T)$ contains  $B_0$. 
If $m\neq n_j^{(i)}$ for all $j=1,\ldots,\delta$, then $m$ is a disconnecting node,
since  $\pi(m)=q_i$ is a node of $B$ but $\tau(m)$ is not a node of $C$. 
Moreover, in this case  and $\pi(T)$ is the tail associated to $q_i$ not containing $B_i$, and thus containing $B_0$.

Furthermore, 
since $\pi$ is quasi admissible, then the restriction of $\pi$ to the component $T'$ of $T$ containing $m$ is finite of degree
$$\deg(\pi|_{T'})\geq e_{\pi_i}(m)$$ 
ramified to order $e_{\pi_i}(m)$ at $m$. 
Let  $B'=\pi(T')$ and note that $q_i\in B'$. 
Let $q\in B'$ be a node of $B$, distinct from $q_i$.
Since $\pi$ is quasi admissible, the points in $\pi^{-1}(q)$ are all nodes of $C'$. 
Let $B''$ be the component of $B$ such that $q$ is the node between $B'$ and $B''$.
Then $\pi^{-1}(B'')\cap T$ contains the nodes 
$\pi^{-1}(q)\cap T'$. Moreover, 
for each $m'\in\pi^{-1}(q)\cap T'$
the set $\pi^{-1}(B'')\cap T$ contains the component $T''_{m'}$ of $T$ such that $m'$ is the node between $T'$ and $T''_{m'}$. 
Hence
\begin{align*}
\deg(\pi|_{\pi^{-1}(B'')\cap T})
&\geq 
\sum_{m'\in\pi^{-1}(q)\cap T'} \deg(\pi|_{T''_{m'}})\\
&\geq 
\sum_{m'\in\pi^{-1}(q)\cap T'} e_{\pi|_{T'}}(m'')
=
\deg(\pi|_{T'})
\geq e_{\pi_i}(m).
\end{align*}
Proceeding in this manner, the result follows.
\end{proof}

\begin{proposition}\label{prop:converse}
Let $C$ be a nodal curve with smooth components $C_1$ and $C_2$ and nodes $n_1,\ldots,n_\delta$.  Let  $\pi\: C'\rightarrow B$ be a quasi admissible cover of degree $k$, where $C'$ is stably equivalent to $C$. 

For $i=1,2$,
let $k_i$ be the degree of $\pi_i=\pi|_{C_i}$, and set $B_i=\pi(C_i)$ and
$$Q_i=\{\pi_i(n_j^{(i)})\ |\ j=1,\ldots,\delta\},$$
where $n_j^{(i)}\in C_i$ is the branch of $n_j$ in $C_i$,
for each $j=1,\ldots, \delta$.

We have
\begin{enumerate}[(i)]
\item \label{item:b1=b2} if $B_1=B_2$ then 
$k\geq k_1+k_2;$

\item \label{item:qi}
If $B_1\neq B_2$
then
$k\geq k_i+|Q_i|-1,$
 for $i=1,2$; 

\item \label{item:main}  If $B_1\neq B_2$ and  $|Q_1|=|Q_2|=1$ then 
 $$k\geq k_1+k_2-(e_1+\ldots+e_\delta),$$ 
where and $e_j=\min\{e_{\pi_1}(n_j^{(1)}),e_{\pi_2}(n_j^{(2)})\}$, for $j=1,\ldots,\delta$.
\end{enumerate} 
\end{proposition}
\begin{proof}
%
 To see (\ref{item:b1=b2}) it is enough to note that,
for a general point $q\in B_i$, $\pi_i^{-1}(q)$ consists of $k_i$ points in $C_i$, for $i=1,2$. 
These two sets are disjoint and, since $B_1\neq B_2$, 
both are contained in $\pi^{-1}(q)$, 
showing that $k\geq k_1+k_2$.


Now assume $B_1\neq B_2$. 
Let $q_1\in Q_1$, $q_2\in Q_2$ and $B_0$ be as in Lemma \ref{lem:Tj}. 

Let's show (\ref{item:qi}) for $i=1$. The proof for $i=2$ is analogous.
If $|Q_1|=1$ there is nothing to show, as $\pi_1=\pi|_{C_1}$. Now assume $|Q_1|\neq 1$. 
For each $q_1'\in Q_1$ with $q_1'\neq q_1$
choose  $j(q_1')\in\{1,\ldots ,\delta\}$ such that 
$q_1'=\pi(n_{j(q_1')}^{(1)})$ and 
set $q_2'=\pi(n_{j(q_1')}^{(2)})$. 
Note that $q_1'\neq q_2'$ since otherwise $B_1$ and $B_2$ would meet at this point, contradicting the 
uniqueness of $q_1$ and $q_2$.
In particular, this implies that 
the branches $n_{j(q_1')}^{(1)}$ and $n_{j(q_1')}^{(2)}$ do not meet in $C'$ and thus
 $T_{j(q_1')}=\tau^{-1}(n_{j(q_1')})$ is a  rational chain in $C'$, where $\tau\: C'\rightarrow C$ is the contraction map. 
Now, $\pi(T_{j(q_1')})$ 
contains $q_1'$, and by Lemma \ref{lem:Tj},  it also contains $q_1$.
Since $\pi(T_{j(q_1')})$  is connected and
$q_1\neq q_1'$ then $\pi(T_{j(q_1')})$ must contain $B_1$.
Hence, for a general point $q\in B_1$ the preimage $\pi^{-1}(q)$ contains 
$k_1$ points in $C_1$ (corresponding to $\pi_1^{-1}(q)$),
 and at least one point of $T_{j(q_1')}$ for every $q_1'\in Q_1$ such that $q_1'\neq q_1$, showing that 
$$
k\geq k_{{1}}
+|Q_1|-1
.$$


Finally we show (\ref{item:main}). 
Since $|Q_i|=1$ then $Q_i=\{q_i\}$ 
and we have $\pi_i(n_j^{(i)})=q_i$ for all $i=1,2$ and $j=1,\ldots,\delta$. 
Then  
$$
\pi_i^{-1}(q_i)=e_{\pi_i}(n_1^{(i)})n^{(i)}_1+\ldots+e_{\pi_i}(n_\delta^{(i)})n^{(i)}_\delta
+e_{\pi_i}(m_1^{(i)})m_1^{(i)}+\ldots+e_{\pi_i}(m_{r_i}^{(i)})m_{r_i}^{(i)},
$$ 
for some $m_1^{(i)},\ldots,m_{r_i}^{(i)}\in C_i$.
Note that
\begin{equation}\label{eq:index}
e_{\pi_i}(m_1^{(i)})+\ldots+e_{\pi_i}(m_{r_i}^{(i)})=k_i-(e_{\pi_i}(n_1^{(i)})+\ldots+e_{\pi_i}(n_\delta^{(i)})).
\end{equation}

Since $m_l^{(i)}\neq n_j^{(i)}$ for every $j=1,\ldots,\delta$, and $\pi(m_l^{(i)})$ is a node of $B$ then $m_l^{(i)}$ must be a disconnecting node of $C'$ and 
thus $T_l^{(i)}=\tau^{-1}(\tau(m_l^{(i)}))$ is rational tail,
for each $i=1,2$ and  $l=1,\ldots,r_i$.
Note that $\pi$ maps  $T_l^{(i)}$ onto the tail of $B$ associated to $q_i$ 
not containing $B_i$,
and hence containing $B_{i'}$ for $i'\neq i$, $i'\in\{1,2\}$.
Moreover, by Lemma \ref{lem:Tj}, for each component 
$B'$ of $B$, either $T_l^{(i)}\cap \pi^{-1}(B')$ is finite or  
the restriction of $\pi$ to $T_l^{(i)}\cap \pi^{-1}(B')$ has degree at least $e_{\pi_i}(m_l^{(i)})$.
Since $\pi$ is quasi admissible, the restriction is 
ramified to order $e_{\pi_i}(m_l^{(i)})$ at  $m_l^{(i)}$
on the component of $T_l^{(i)}$ containing $m_l^{(i)}$.

Now, for each $j=1,\ldots, \delta$, the preimage $\tau^{-1}(n_j)$ consists of either a point or a rational chain. If $\tau^{-1}(n_{j_0})$ is a point for some $j_0=1,\ldots, \delta$, then $C_1$ and $C_2$ meet at $n_{j_0}^{(1)}=n_{j_0}^{(2)}=\tau^{-1}(n_{j_0})$ in $C'$. But this implies that $q_1=q_2$  and hence $n_j^{(1)}=n_j^{(2)}$ for all $j=1,\ldots,\delta$. 
Since $\pi$ is quasi admissible, this implies that
\begin{equation}\label{eq:e1=e2}
e_{\pi_1}(n_j^{(1)})=e_{\pi_2}(n_j^{(2)})=e_j
\end{equation}
 for all $j=1,\ldots,\delta$.
Hence for a general point $q\in B_1$, among the $k$  points of  $\pi^{-1}(q)$, 
there are 
$k_1$ points lying in $C_1$ and 
at least $e_{\pi_2}(m_l^{(2)})$ points lying in $T_l^{(2)}$, for  $l=1,\ldots,r_2$, and we have
$$k\geq k_1+ e_{\pi_2}(m_1^{(2)})+\ldots+e_{\pi_2}(m_{r_2}^{(2)}).$$
Thus, from (\ref{eq:index})  and (\ref{eq:e1=e2}) we get
$$k\geq k_1+k_2-(e_1+\ldots+e_\delta).$$

Now assume  $T_j=\tau^{-1}(n_j)$ is a rational chain for every $j=1,\ldots,\delta$. 
Then by Lemma \ref{lem:Tj}, $\pi(T_j)$ contains $B_0$, for $j=1,\ldots,\delta$. 
Note that the image $\pi(T_l^{(i)})$ also contains $B_0$, for  $l=1,\ldots,r_i$ and $i=1,2$.

By Lemma \ref{lem:Tj}, 
for each component 
$B'$ of $B_0$,
the restriction of $\pi$ to $T_j\cap \pi^{-1}(B')$ has degree at least 
$\max\{e_{\pi_1}(n_j^{(1)}), e_{\pi_2}(n_j^{(2)})\}$,
and
the restriction of $\pi$ to $T_l^{(i)}\cap \pi^{-1}(B')$ has degree at least 
$e_{\pi_i}(m_l^{(i)})$ ,
 for every $i=1,2$, $j=1,\ldots,\delta$ and $l=1,\ldots,r_i$.
Therefore, for a general point $q\in B_0$, among the $k$  points of  $\pi^{-1}(q)$, 
there are 
at least $\max\{e_{\pi_1}(n_j^{(1)}), e_{\pi_2}(n_j^{(2)})\}$ points lying in $T_j$, for $j=1,\ldots,\delta$, and 
at least $e_{\pi_i}(m_l^{(i)})$ points lying in $T_l^{(i)}$, for  $l=1,\ldots,r_i$ and $i=1,2$. Hence, from (\ref{eq:index}), we have
\begin{eqnarray*}
k&\geq &\sum_{j=1}^{\delta}\max\{e_{\pi_1}(n_j^{(1)}), e_{\pi_2}(n_j^{(2)})\}
+\sum_{l=1}^{r_1}e_{\pi_1}(m_l^{(1)})+\sum_{l=1}^{r_2}e_{\pi_2}(m_l^{(2)})\\
&=&\sum_{j=1}^{\delta}\max\{e_{\pi_1}(n_j^{(1)}), e_{\pi_2}(n_j^{(2)})\} 
+k_1-\sum_{j=1}^{\delta}e_{\pi_1}(n_j^{(1)}) +k_2-\sum_{j=1}^{\delta}e_{\pi_2}(n_j^{(1)})\\
&=&k_1+k_2-(e_1+\ldots+e_\delta).
\end{eqnarray*}
\end{proof}

In particular, for a two-component stable curve of compact type we have:

\begin{theorem}\label{thm:goncharcpt}
Let $C$ be a stable curve with smooth  components $C_1$ and $C_2$ and a single node $n$.
For $i=1,2$, let  $n^{(i)}\in C_i$ be the branch of $n$ in $C_i$ and 
let $e_i$ be the maximum ramification index of  a degree-$\gon(C_i)$ map $C_i\rightarrow \mathbb{P}^1$ at $n^{(i)}$.
Then $\gon(C)=\gon(C_1)+\gon(C_2)-e$, where
$e=\min\{e_1,e_2\}$.
\end{theorem}
\begin{proof}
By Corollary \ref{cor:gonchar}, $\gon(C)\leq\gon(C_1)+\gon(C_2)-e$.

Converselly, let $\pi\: C'\rightarrow B$ be an admissible cover of degree $\deg(\pi)=\gon(C)$, for some $C'$ stably equivalent to $C$. 
Let $\pi_i=\pi|_{C_i}$, $B_i=\pi(C_i)$, $k_i=\deg(\pi_i)$ and $Q_i$ be as in Lemma  \ref{lem:Tj}, for $i=1,2$.
First note that, by Proposition \ref{prop:converse} (\ref{item:b1=b2}) we have $B_1\neq B_2$ since otherwise $\deg(\pi)$ would be at least $\gon(C_1)+\gon(C_2)$. 
Now, since $C$ has only one node, $|Q_1|=|Q_2|=1$ and by Proposition \ref{prop:converse} (\ref{item:main}), 
$$\gon(C)=\deg(\pi)\geq k_1+k_2-e(\pi_1,\pi_2)\geq \gon(C_1)+\gon(C_2)-e$$ 
and the result follows.
\end{proof}

\section{Hyperelliptic and trigonal two-component curves}\label{sec:hyptrig}

A stable curve is \emph{hyperelliptic} if $\gon(C)=2$ and is \emph{trigonal} if $\gon(C)=3$. 
As a consequence of the resuts of the previous sections, we can characterize hyperelliptic and trigonal  stable curves having two smooth non rational components.  First we need some lemmas.

\begin{lemma}\label{lem:goncomp}
Let $C$ be a stable curve and let $C_i$ be a smooth component of $C$.
Then $\gon(C)\geq \gon(C_i)$.
\end{lemma}
\begin{proof}
If $C_i$ is rational, there is nothing to prove, since in this case $\gon(C_i)=1$.

Now assume $C_i$ non rational. If $\pi\: C'\rightarrow B$ is an admissible cover of degree $\gon(C)$ with $C'$ stably equivalent to $C$, then $C'$ contains $C_i$ as a component, and the restriction of $\pi$ to $C_i$ is quasi admissible, by Lemma \ref{lem:ratldeg2}. Thus, by Theorem \ref{thm:quasitoadmiss}, $C_i$ is $gon(C)$-gonal and  the result follow.
\end{proof}

\begin{lemma}\label{lem:degdelta}
Let $C$ be a stable curve  with two smooth non rational components $C_1$ and $C_2$ and nodes $n_1,\ldots,n_\delta$. 
Let $\pi\: C'\rightarrow B$ be a quasi admissible cover, where $C'$ is stably equivalent to $C$. 

Let $\pi_i=\pi|_{C_i}$, $q_i\in Q_i$ and $n_j^{(i)}\in C_i$  be as in Lemma \ref{lem:Tj}, 
for $i=1,2$ and $j=1,\ldots,\delta$. If $\deg(\pi_i)=\deg(\pi)$ for some $i=1,2$ then 
$$\deg(\pi)\geq\delta 
\quad\text{ and }\quad
e_{\pi_i}(n_j^{(i)})\geq e_{\pi_{i'}}(n_j^{(i')}),$$ 
for all $j\in\{1,\ldots,\delta\}$ such that 
$\pi_1(n_j^{(1)})=q_1$
and $\pi_{2}(n_j^{(2)})=q_{2}$, 
where $\{i,i'\}=\{1,2\}$.
\end{lemma}
\begin{proof}
Without loss of generality, assume $\deg(\pi_1)=\deg(\pi)$. 
Let $B_1=\pi(C_1)$ and $B_2=\pi(C_2)$.

By Proposition \ref{prop:converse} (\ref{item:b1=b2}) we have $B_1\neq B_2$ and thus, by  Proposition \ref{prop:converse} (\ref{item:qi}) we have $|Q_1|=1$, so $Q_1=\{q_1\}$. 
But then 
$\pi_1(n_1^{(1)})=\ldots=\pi_1(n_\delta^{(1)})=q_1$ which implies that $$\delta\leq \deg(\pi_1)=\deg(\pi),$$
showing the first assertion.

To show the second assertion, let $B_0$ be as 
in Lemma \ref{lem:Tj},
that is, 
either $q_1=q_2$ or  
$B_0$ is  the intersection between  the tail  of $B$ associated to $q_1$ not containing $B_1$ with  the  tail  of $B$ associated to $q_2$ not containing $B_2$.
Let $j\in\{1,\ldots,\delta\}$ such that $\pi_1(n_j^{(1)})=q_1$
and $\pi_{2}(n_j^{(2)})=q_{2}$.

We proceed by contradiction, so assume 
$e_{\pi_1}(n_j^{(1)})< e_{\pi_{2}}(n_j^{(2)})$. Since $\pi$ is quasi  admissible, this implies that $n_j^{(1)}\neq n_j^{(2)}$ and thus
$q_1\neq q_2$. In particular,  $T_j=\tau^{-1}(n_j)$ is a rational chain in $C'$ 
and $\pi(T_j)$ contains $B_0$,
where $\tau\: C'\rightarrow C$ is the contraction map.
Now,  by Lemma \ref{lem:Tj},
for each component $B'$ of $B_0$, either $T_j\cap \pi^{-1}(B')$ is finite or  
the restriction of $\pi$ to $T_j\cap \pi^{-1}(B')$ 
has degree at least  $\max\{e_{\pi_1}(n_j^{(1)}),\, e_{\pi_{2}}(n_j^{(2)})\}$, 

In particular if $B'$ is the component of $B_0$ containing $q_1$,
since $e_{\pi_1}(n_j^{(1)})< e_{\pi_{2}}(n_j^{(2)})$, the restriction of $\pi$ to 
$T_j\cap \pi^{-1}(B')$ has degree
$$\deg(\pi|_{T_j\cap \pi^{-1}(B')})\geq e_{\pi_2}(n_j^{(2)})
>e_{\pi_1}(n_j^{(1)})$$ 
and is ramified to order $e_{\pi_1}(n_j^{(1)})$ at $n_j^{(1)}$.
Hence there exists at least one point $m \in T_j$ with $m\neq n_j^{(1)}$ such that $\pi(m)=q_1$. 
Since $\pi$ is quasi admissible and  $q_1$ is a node of $B$, then $m$ must be a node of $C'$. 
Furthermore, since $C'$ is stably equivalent to $C$ then $m$ must be a disconnecting node of $C'$. 
Since $\pi$ maps $m$ to $q_1$ then, for one of the tails $T'$ associated to $m$, the image $\pi(T')$ contains $B_1$.
Then,
for a general point $q\in B_1$, the preimage $\pi^{-1}(q)$  contains $\deg(\pi_1)$ points of $C_1$ and at least one point of $T'$, contradicting the fact that $\deg(\pi_1)=\deg(\pi)$. 
\end{proof}

To simplify the statements, if $\pi\:C\rightarrow \Pbb^1$ is a rational map from a smooth curve $C$, we 
say that a set of points $p_1,\ldots,p_n$ on $C$ is \emph{conjugated under $\pi$} if 
$\pi(p_1)=\ldots=\pi(p_n)$.

\begin{theorem}\label{thm:hyp}
Let $C$ be a stable curve  with two smooth non rational components. 
Then $C$ is hyperelliptic if and only if both components are hyperelliptic and one of the following holds:
\begin{enumerate}[(i)]
\item $C$ has one node and the branches of the node in each component have ramification index 2 under a degree-2 map;

\item $C$ has two nodes and the branches of the nodes in each component are
conjugated points of ramification index 1 under a degree-2 map.
\end{enumerate}
\end{theorem}
\begin{proof}
Let $C_1$ and $C_2$ be the components of $C$.  If follows from Lemma \ref{lem:goncomp} and Corollary \ref{cor:gonchar} that, if $C_1$ and $C_2$ are hyperelliptic and either (i) or (ii) hold, then $C$ is also hyperelliptic.

Now assume $C$ is hyperelliptic. 
By Lemma \ref{lem:goncomp}, $C_1$ and $C_2$ are both hyperelliptic, since they are  both non rational by hypothesis.
Let $\pi\: C'\rightarrow B$ be an admissible cover of degree $\deg(\pi)=2$, for some $C'$ stably equivalent to $C$. As in Lemma \ref{lem:Tj}, let $\pi_i=\pi|_{C_i}$, for $i=1,2$. 
Since $C_i$ is non rational and $\deg(\pi_i)\leq \deg(\pi)=2$, we must have $\deg(\pi_i)=2$. 
By Lemma \ref{lem:degdelta} this implies that $C$ has at most two nodes.

If $C$ has one node then, in the notation of Theorem \ref{thm:goncharcpt}, we must have $e=2$  and thus the  branch of the node in each component  must have ramification index 2.
Finally assume $C$ has two nodes. By Proposition \ref{prop:converse} (\ref{item:b1=b2}) we see that $\pi(C_1)\neq \pi(C_2)$. Hence, by Proposition \ref{prop:converse} (\ref{item:qi}), the branches of the nodes must be conjugated points under $\pi_1$ and $\pi_2$
and, by Proposition \ref{prop:converse} (\ref{item:main}), they must have ramification index 1.
\end{proof}

\begin{lemma}\label{lem:gluegon}
Let $C$ be a stable curve  with two smooth non rational components $C_1$ and $C_2$. 
Let $n$ be a node of $C$ and let $\tilde C$ be the normalization of $C$ at $n$. 
Let $\tilde\pi\:\tilde C'\rightarrow \tilde B$ be an admissible cover,  where $\tilde C'$ is stably equivalent to $\tilde C$. 

Let $q_1,q_2\in\tilde B$ be as in Lemma \ref{lem:Tj} and
 let $n^{(i)}\in C_i$ be the branch of $n$ at $C_i$, for $i=1,2$.

\begin{enumerate}[(i)]
\item Assume that $\tilde C$ is connected, $\tilde\pi(C_1)\neq\tilde\pi(C_2)$, and
$$\tilde \pi(n^{(i)})=q_i
\quad\text{ and  }\quad
\tilde \pi(n^{({i'})})\neq q_{i'}$$
for $\{i,i'\}=\{1,2\}$.
Then there exists an admissible cover 
$\pi\:C'\rightarrow B$ of degree $\deg(\pi)=\deg(\tilde\pi)$ with $C'$ stably equivalent to $C$, containing $\tilde C'$ as a subcurve and such that $\pi|_{\tilde C'}=\tilde \pi$. 

\item Assume that $\tilde C$ is connected, $\tilde\pi(C_1)\neq\tilde\pi(C_2)$, and
$$\tilde \pi(n^{(1)})\neq q_1
\quad\text{ and  }\quad
\tilde \pi(n^{({2})})\neq  q_{2}.$$
Then there exists an admissible cover 
$\pi\:C'\rightarrow B$ of degree $\deg(\pi)=\deg(\tilde\pi)+1$ with $C'$ stably equivalent to $C$, containing $\tilde C'$ as a subcurve and such that $\pi|_{\tilde C'}=\tilde \pi$.

\end{enumerate}
\end{lemma}
\begin{proof}
For $i=1,2$, let 
$T_i=\tilde\tau^{-1}(\tilde\tau(n^{(i)}))$ where $\tilde\tau\:\tilde C'\rightarrow \tilde C$ is the contraction map, for $i=1,2$. 
Since $\tilde \pi$ is admissible and $\tilde\tau(n^{(i)})$ is not a node of $\tilde C$, 
then either $T_i$ is a rational tail
or $T_i=n^{(i)}$ is a smooth point of $\tilde C$, for $i=1,2$.

To show (ii),
let ${n'}^{(1)}\in T_1$ and  ${n'}^{(2)}\in T_2$ be smooth points
such that 
$\tilde \pi({n'}^{(1)})\neq 
\tilde \pi({n'}^{({2})})$.
Then
the result follows directly from \cite[Theorem 3.4 (a)]{paper1}.

Now we show (i). Without loss of generality, we assume $i=1$ and $i'=2$.
Since $\tilde \pi$ is admissible and $\tilde \pi(n^{(1)})=q_1$ is a node of $\tilde B$ 
then $n^{(1)}$ is a node of $\tilde C'$ and hence $T_1$ is a rational tail.
Moreover, by Lemma \ref{lem:Tj}, $\pi(T_1)$ contains $q_1$ and $q_2$.
Since $\tilde\pi(T_1)$ is the tail associated to $\tilde \pi(n^{(1)})=q_1$
not containing 
$\tilde\pi(C_{1})$, then in particular
$\tilde\pi(T_1)$ contains  $\tilde\pi(T_{2})$.
Let ${n'}^{({2})}$ be any smooth point of $T_{2}$, 
and let ${n'}^{(1)}$ be a point of $T_1$ such that 
$\tilde \pi({n'}^{(1)})=\tilde \pi({n'}^{({2})})$. 
The result thus follows from
\cite[Theorem 3.4(b)]{paper1}.
\end{proof}

\begin{theorem}\label{thm:trig}
Let $C$ be a stable curve  with two smooth non rational components. 
Then $C$ is trigonal if and only if the components are either hyperelliptic or trigonal and one of the following cases hold:
\begin{enumerate}[(i)]
\item $C$ has one node, and the branches of the node have 
	\begin{enumerate}[(a)]
	\item  ramification index 1 under a degree-2 map for one component, and 
		ramification index 1 or 2 under a degree-2 map for the other component;

	\item ramification index 2 under a degree-2 map for one  component, and 
		ramification index 2 or 3 under a degree-3 map for the other component;


	\item ramification index 3 under a degree-3 map, for both components.
	\end{enumerate}

\item $C$ has two nodes, and the branches of the nodes are
	\begin{enumerate}[(a)]

	\item  conjugated points of ramification index 1 under a degree-2 map for one component, and
		conjugated points of ramification index 1 or 2 under a degree-3 map for the other component;


	\item  conjugated points of ramification indexes 1 and 2 under a degree-3 map for both components,  such that the branches corresponding to the same node have the same ramification index;

	\item non conjugated points under a degree-2 map for one component, at least one of which has ramification index 2, and
		conjugated points of ramification indexes 1 and 2 under a degree-3 map for the other  component,
		such that the branch having ramification index 2 on the second component correspond to the same node as  a branch having ramification index 2 on the first component.

	\item conjugated points of ramification index 1 under a degree-2 map for one component, and
		non conjugated points of ramification indexes 1 or 2  under a degree-2 map for the other component;

	\item non conjugated points, at least one of which has ramification index 2 under a degree-2 map, for both components, 
such that a branch having ramification index 2 on the first component correspond to the same node as  a branch having ramification index 2 on the second component. 

	\end{enumerate}

\item $C$ has three nodes, and the branches of the nodes are 
	\begin{enumerate}[(a)]
	\item  conjugated points of ramification index 1 under a degree-3 map, for both components;

	\item two conjugated points of ramification index 1 and one non conjugated point of ramification index 1 or 2 under a degree-2 map, for one component, and
conjugated points of ramification index 1 under a degree-3 map for the other component;

	\item  two conjugated points of ramification index 1 and one non conjugated point of ramification index 1 or 2 under a degree-2 map, for both components.
	\end{enumerate}
\end{enumerate}
\end{theorem}
\begin{proof}
Let $C_1$ and $C_2$ be the components of $C$.  

Assume first that (i), (ii) or (iii) hold. Then by Lemma \ref{lem:goncomp}, $\gon(C)\geq 2$. But, by Theorem \ref{thm:hyp}, $\gon(C)\neq 2$ and thus we have $\gon(C)\geq 3$. 
If the branches of all the nodes are conjugated 
that is, if (i), (ii)(a), (ii)(b) or (iii)(a) hold,
then by Corollary \ref{cor:gonchar} 
$\gon(C)\leq 3$, showing that $C$ is trigonal

Let's now examine the remaining cases.

Let $C$ be as in (ii)(c). Let $n_1$ and $n_2$ be the nodes of $C$ and let $n_j^{(i)}\in C_i$ be the branch of $n_j$ in $C_i$, for $i=1,2$ and $j=1,2$.
Without loss of generality, assume that $n_1^{(1)}$ and $n_2^{(1)}$ are non conjugated and $n_1^{(2)}$ and $n_2^{(2)}$ are conjugated, and that
$$e_{\pi_1}(n_1^{(1)})=
e_{\pi_2}(n_1^{(2)})=2,
\quad\text{ and} \quad
e_{\pi_2}(n_2^{(2)})=1.$$
If $\tilde C$ is the normalization of $C$ at $n_2$ then $\tilde C$ is as in (i)(b).
Then there is an admissible cover 
$\tilde\pi\:\tilde C'\rightarrow \tilde B$ of degree $\deg(\tilde\pi)=3$ with $\tilde C'$ stably equivalent to $\tilde C$.
Note that that $q_1=\tilde\pi(n_1^{(1)})$ and $q_2=\tilde\pi(n_1^{(2)})$ are as in Lemma \ref{lem:Tj}, and by hypothesis
$\tilde\pi(n_2^{(1)})\neq q_1$ and $\tilde\pi(n_2^{(2)})=q_2$. 
Then by Lemma \ref{lem:gluegon} (i), $\gon(C)\leq 3$, showing that $C$ is trigonal.

Let $C$ be as in (ii)(d). Let $n_1$ and $n_2$ be the nodes of $C$ and let $n_j^{(i)}\in C_i$ be the branch of $n_j$ in $C_i$, for $i=1,2$ and $j=1,2$.
Without loss of generality, assume that $n_1^{(1)}$ and $n_2^{(1)}$ are conjugated and $n_1^{(2)}$ and $n_2^{(2)}$ are non conjugated.
If $\tilde C$ is the normalization of $C$ at $n_2$ then $\tilde C$ is as in (i)(a).
Then there is an admissible cover 
$\tilde\pi\:\tilde C'\rightarrow \tilde B$ of degree $\deg(\tilde\pi)=3$ with $\tilde C'$ stably equivalent to $\tilde C$.
Note that that $q_1=\tilde\pi(n_1^{(1)})$ and $q_2=\tilde\pi(n_1^{(2)})$ are as in Lemma \ref{lem:Tj}, and by hypothesis
$\tilde\pi(n_2^{(1)})=q_1$ and $\tilde\pi(n_2^{(2)})\neq q_2$. 
Then by Lemma \ref{lem:gluegon} (i), $\gon(C)\leq 3$, showing that $C$ is trigonal.

Let $C$ be as in (ii)(e). Let $n_1$ and $n_2$ be the nodes of $C$ and let $n_j^{(i)}\in C_i$ be the branch of $n_j$ in $C_i$, for $i=1,2$ and $j=1,2$.
Without loss of generality, assume that 
$e_{\pi_1}(n_1^{(1)})=
e_{\pi_2}(n_1^{(2)})=2$.
If $\tilde C$ is the normalization of $C$ at $n_2$ then $\tilde C$ is as in Theorem \ref{thm:hyp} (i).
Then there is an admissible cover 
$\tilde\pi\:\tilde C'\rightarrow \tilde B$ of degree $\deg(\tilde\pi)=2$ with $\tilde C'$ stably equivalent to $\tilde C$.
Note that that $q_1=\tilde\pi(n_1^{(1)})$ and $q_2=\tilde\pi(n_1^{(2)})$ are as in Lemma \ref{lem:Tj}, and by hypothesis
$\tilde\pi(n_2^{(1)})\neq q_1$ and $\tilde\pi(n_2^{(2)})\neq q_2$. 
Then by Lemma \ref{lem:gluegon} (ii), $\gon(C)\leq 3$, showing that $C$ is trigonal.

Let $C$ be as in (iii)(b). Let $n_1$, $n_2$ and $n_3$ be the nodes of $C$ and let $n_j^{(i)}\in C_i$ be the branch of $n_j$ in $C_i$, for $i=1,2$ and $j=1,2,3$.
Without loss of generality, assume that 
$n_1^{(1)}, n_2^{(1)}$ are conjugated and $n_3^{(1)}$ is non conjugated, and that
$n_1^{(2)},n_2^{(2)},n_3^{(2)}$ are conjugated.
If $\tilde C$ is the normalization of $C$ at $n_3$ then $\tilde C$ is as in (ii)(a).
Then there is an admissible cover 
$\tilde\pi\:\tilde C'\rightarrow \tilde B$ of degree $\deg(\tilde\pi)=3$ with $\tilde C'$ stably equivalent to $\tilde C$.
Note that that $q_1=\tilde\pi(n_1^{(1)})$ and $q_2=\tilde\pi(n_1^{(2)})$ are as in Lemma \ref{lem:Tj}, and by hypothesis
$\tilde\pi(n_3^{(1)})\neq q_1$ and $\tilde\pi(n_3^{(2)})=q_2$. 
Then by Lemma \ref{lem:gluegon} (i), $\gon(C)\leq 3$, showing that $C$ is trigonal.

Finally, 
let $C$ be as in (iii)(c). Let $n_1$, $n_2$ and $n_3$ be the nodes of $C$ and let $n_j^{(i)}\in C_i$ be the branch of $n_j$ in $C_i$, for $i=1,2$ and $j=1,2,3$.
Without loss of generality, assume that 
$n_1^{(i)}, n_2^{(i)}$ are conjugated and $n_3^{(i)}$ is non conjugated, for $i=1,2$.
If $\tilde C$ is the normalization of $C$ at $n_3$ then $\tilde C$ is as in Theorem \ref{thm:hyp} (ii).
Then there is an admissible cover 
$\tilde\pi\:\tilde C'\rightarrow \tilde B$ of degree $\deg(\tilde\pi)=2$ with $\tilde C'$ stably equivalent to $\tilde C$.
Note that that $q_1=\tilde\pi(n_1^{(1)})$ and $q_2=\tilde\pi(n_1^{(2)})$ are as in Lemma \ref{lem:Tj}, and by hypothesis
$\tilde\pi(n_3^{(1)})\neq q_1$ and $\tilde\pi(n_3^{(2)})\neq q_2$. 
Then by Lemma \ref{lem:gluegon} (ii), $\gon(C)\leq 3$, showing that $C$ is trigonal.

\medskip

Now assume $C$ is trigonal. By Lemma \ref{lem:goncomp}, $C_1$ and $C_2$ are either hyperelliptic or trigonal, since they are  both non rational by hypothesis.
Let $\pi\: C'\rightarrow B$ be an admissible cover of degree $\deg(\pi)=3$, for some $C'$ stably equivalent to $C$. 
As in Lemma \ref{lem:Tj}, let $\pi_i=\pi|_{C_i}$ and $B_i=\pi(C_i)$, for $i=1,2$. 
Since $C_i$ is non rational and $\deg(\pi_i)\leq \deg(\pi)=3$, we must have $\deg(\pi_i)=2$ or $\deg(\pi_i)=3$, for $i=1,2$. 
In particular, by Proposition \ref{prop:converse} (\ref{item:b1=b2}), we must have $B_1\neq B_2$.


Let $n_1,\ldots,n_\delta$ be the nodes of $C$.
For $j=1,\ldots,\delta$,  let 
$n_j^{(i)}\in C_i$  be the branch of $n_j$ in $C_i$ and
$$e_j=\min\{ e_{\pi_1}(n_j^{(1)}),\,e_{\pi_2}(n_j^{(2)})\},
$$ 
and let
$Q_i$ and $q_i$ be as in Lemma \ref{lem:Tj}, 
for $i=1,2$.
We will consider four cases, depending on $\deg(\pi_i)$.

CASE 1. $\deg(\pi_1)=\deg(\pi_2)=3$: 
In this case, both components are either hyperelliptic or trigonal and
by Lemma \ref{lem:degdelta} we have $\delta\leq 3$. Moreover, by Proposition \ref{prop:converse}  (\ref{item:qi}), the branches of the nodes must be conjugated by $\pi_1$ and $\pi_2$.
Now, $e_j\leq e_{\pi_i}(n_j^{(i)})$ and thus 
$$\sum_{j=1}^\delta e_j\leq \sum_{j=1}^\delta e_{\pi_i}(n_j^{(i)}) =\deg(\pi_i)=3.$$
Therefore, Proposition \ref{prop:converse} (\ref{item:main}) gives
$\sum_{j=1}^\delta e_j=3$
and we get (i)(c), (ii)(b) and (iii)(a).

\smallskip

CASE 2. $\deg(\pi_1)=2$ and $\deg(\pi_2)=3$: 
In this case, the first component is hyperelliptic and the second is either hyperelliptic or trigonal. 
Again $\delta\leq 3$, by Lemma \ref{lem:degdelta}.
Moreover, by Proposition \ref{prop:converse} (\ref{item:qi}), 
we have $|Q_1|\leq 2$ and $|Q_2|=1$, so
the branches of the nodes must be conjugated by $\pi_2$, although not necessarely by $\pi_1$. 

If $|Q_1|=1$ then the branches of the nodes must also be conjugated by $\pi_1$, 
and in particular we have $\delta\leq 2$.
Now, since $e_j\leq e_{\pi_1}(n_j^{(1)})$ we have
$$\sum_{j=1}^\delta e_j\leq \sum_{j=1}^\delta e_{\pi_1}(n_j^{(1)}) =\deg(\pi_1)=2$$
and thus  Proposition \ref{prop:converse} (\ref{item:main}) gives
$\sum_{j=1}^\delta e_j=2$.
If $\delta=1$, we get (i)(b) and (i)(c). If $\delta=2$ then $e_1=e_2=1$ and we get (ii)(a).


Now assume $|Q_1|=2$, say $Q_1=\{q_1,\,q_1'\}$. In particular $\delta\geq2$ and thus we have $\delta=2$ or $\delta=3$. 
Note that 
\begin{equation}\label{eq:aux1}
\sum_{n_j^{(1)}\in\pi^{-1}_1(q_1)} e_{\pi_1}(n_j^{(1)})=2.
\end{equation}
Indeed, if that is not the case then there exists $m\in C'$ such that $\pi_1(m)=q_1$ and $m\neq n_j^{(1)}$ for every $j=1,\ldots,\delta$. Then, as in the proof of Lemma \ref{lem:Tj}, 
 $m$ must be a disconnecting node of $C'$ and
 we let $T$ be the tail of $C'$ associated to $m$ not containing $C_1$. 
Then $\pi(T)$ is the tail of $B$ associated to $q_1$ not containing $B_1$,
since  $\pi$ is quasi admissible and  maps $m$ and $C_1$ to $q_1$ and $B_1$, respectively. 
In particular $\pi(T)$ contains $B_2$ 
and, for a general point $q\in B_2$, the preimage $\pi^{-1}(q)$  contains $\deg(\pi_2)=3$ points of $C_2$ and at least one point of $T$, contradicting the fact that $\deg(\pi)=3$. 

If $\delta=2$ then there is only one branch in the preimage of $\pi_1$, say 
$\pi^{-1}_1(q_1)=\{n_1^{(1)}\}$, 
and by (\ref{eq:aux1}) and Lemma \ref{lem:degdelta} we must have 
$e_{\pi_1}(n_1^{(1)})=e_{\pi_2}(n_1^{(2)})=2.$
In particular this implies that $e_{\pi_2}(n_2^{(2)})=1$, and $e_{\pi_1}(n_2^{(1)})$ can be 1 or 2, and we get (ii)(c).

If $\delta=3$, since $|Q_2|=1$ then 
$\pi_2(n_1^{(2)})
=\pi_2(n_2^{(2)})
=\pi_2(n_3^{(2)})=q_2$ and in particular $e_{\pi_2}(n_j^{(2)})=1$ for $j=1,2,3$. By Lemma \ref{lem:degdelta} this implies that, for $n_j^{(1)}\in\pi_1^{-1}(q_1)$, we have $e_{\pi_1}(n_j^{(1)})=1$. Hence, by (\ref{eq:aux1}),  $\pi_1^{-1}(q_1)$ consists of two points, say 
$\pi_1^{-1}(q_1)=\{n_1^{(1)},\, n_2^{(1)}\}$, both having ramification index 1 under $\pi_1$. For the third node $n_3$, we have that $e_{\pi_2}(n_3^{(2)})=1$, and $e_{\pi_1}(n_3^{(1)})$ can be 1 or 2, and we get (iii)(b).

\smallskip 

CASE 3. $\deg(\pi_1)=3$ and $\deg(\pi_2)=2$. This case is analogous to Case 2.


\smallskip

CASE 4. $\deg(\pi_1)=\deg(\pi_2)=2$: 
In this case, both components are hyperelliptic and, by Proposition \ref{prop:converse} (\ref{item:qi}), we have $|Q_i|\leq 2$ for $i=1,2$.

If $|Q_1|=|Q_2|=1$ then the branches of the nodes are all conjugated under $\pi_1$ and $\pi_2$ and, since  $\deg(\pi_1)=\deg(\pi_2)=2$, we have $\delta\leq 2$. 
Note also that, since $e_j\leq e_{\pi_i}(n_j^{(i)})$, we have
$$\sum_{j=1}^\delta e_j\leq \sum_{j=1}^\delta e_{\pi_i}(n_j^{(i)}) =\deg(\pi_i)=2$$
If $\delta=1$ then by Theorem \ref{thm:hyp}, the branch of the node in at least one of the components must have ramification index 1 and we get (i)(a). If $\delta=2$ then, since the nodes are conjugated, all the  branches must have ramification index 1. But in this case, by Theorem \ref{thm:hyp}, $C$ would be hyperelliptic, contradicting the fact that $C$ is trigonal.

If $|Q_1|=1$ and $|Q_2|=2$, then $Q_1=\{q_1\}$ and $Q_2=\{q_2,q_2'\}$. 
Since $|Q_1|=1$,  the branches of the nodes are conjugated under $\pi_1$ and, since $\deg(\pi_1)=2$, there are at most 2 nodes and the branches of the nodes are conjugated points of ramification index 1 under $\pi_1$. 
On the other hand, since $|Q_2|=2$, we have $\delta\geq2$ showing that $C$ has exactly 2 nodes. 
Moreover, the branches of the nodes are non conjugated points under $\pi_2$ and the ramification index can be 1 or 2 , and we get (ii)(d).

The case where $|Q_1|=2$ and $|Q_2|=1$ is analogous to the previous one.

Finally, consider $|Q_1|=|Q_2|=2$, say $Q_1=\{q_1,q_1'\}$ and $Q_2=\{q_2,q_2'\}$. In particular we have $2\leq \delta\leq 4$.

We first note that if there exist $m\in C_i$ such that $\pi_i(m)=q_i$ and $m\neq n_j^{(i)}$ for all $j=1,\ldots,\delta$ 
then $T=\tau^{-1}(\tau(m))$ is a rational tail associated to $q_i$ containing $B_{i'}$, where $\{i,i'\}=\{1,2\}$.
Moreover, if $\pi_{i'}(n_j^{({i'})})=q_{i'}'$ and  $T_j=\tau^{-1}(n_j)$ then  $\pi(T_j)$
contains $q_{i'}$ and $q_{i'}'$ and hence contains also $B_{i'}$. Thus, for a general point $q\in B_{i'}$, the preimage $\pi^{-1}(q)$ contains two points of $C_{i'}$, at least one point of $T$ and at least one point of $T_j$, contradicting the fact that $\deg(\pi)=3$. 

Note moreover that if $\pi_i(n_{j}^{(i)})=q_i'$ and  $T_{j}=\tau^{-1}(n_{j})$ then  $\pi(T_{j})$
contains $q_i$, since it is connected and contains $\pi_{i'} (n_j^{(i')}) \in B_{i'}$ and $\pi_i(n_{j}^{(i)})\in B_i$, for $\{i,i'\}=\{1,2\}$. Hence, since $\pi(T_{j})$ contains $q_i$ and $q_i'$, it also contains $B_i$. 
Since $\deg(\pi)=3$ and $\deg(\pi_i)=2$, there can be at most one such $j$, that is, 
$$|\pi^{-1}(q_i')|=1,$$
for $i=1,2$. 
In particular this shows that $\delta\leq 3$
and either
$$|\pi^{-1}(q_1)|=|\pi^{-1}(q_2)|=1
\quad\text{ or }\quad
|\pi^{-1}(q_1)|=|\pi^{-1}(q_2)|=2.
$$

Let's examine the possible cases.
If 
$\pi_1^{-1}(q_1)=\{n_j^{(1)}\}$ and
$\pi_2^{-1}(q_2)=\{n_j^{(2)}\}$
for some $j\in\{1,\ldots,\delta\}$, then in particular 
$e_{\pi_i}(n_j^{(i)})=2$ for $i=1,2$.
Since $|\pi^{-1}(q_i')|=1$, 
we get (ii)(e).

If 
$\pi_1^{-1}(q_1)=\{n_{j_1}^{(1)}\}$ and
$\pi_2^{-1}(q_2)=\{n_{j_2}^{(2)}\}$
for some ${j_1},{j_2}\in\{1,\ldots,\delta\}$ with ${j_1}\neq {j_2}$, 
 then in particular $q_1\neq q_2$ and the rational chain $B_0$ defined in Lemma \ref{lem:Tj} exists. 
Moreover, by Lemma \ref{lem:Tj} if
$T_{j_i}=\tau^{-1}(n_{j_i})$
then $\pi(T_{j_i})$ contains $B_0$
and, for every component $B'$ of $B_0$, the restriction of $\pi$ to $T_{j_i}\cap \pi^{-1}(B')$ has degree at least 
$e_{\pi_i}(n_{j_i}^{(i)})=2$, for $i=1,2$.
Thus for a general point $q$ in $B'$, 
the preimage $\pi^{-1}(q)$ contains at least two points of $T_{j_1}$ and at least two points of $T_{j_2}$, contradicting the fact that $\deg(\pi)=3$.  Hence this case does not happen.

If 
$\pi_1^{-1}(q_1)=\{n_{j_1}^{(1)},n_{j_2}^{(1)}\}$
and
$\pi_2^{-1}(q_2)=\{n_{j_1}^{(2)},n_{j_2}^{(2)}\}$
for some ${j_1},{j_2}\in\{1,\ldots,\delta\}$ with ${j_1}\neq {j_2}$, 
then in particular
$e_{\pi_i}(n_{j_1}^{(i)})=e_{\pi_i}(n_{j_2}^{(i)})=1$, for $i=1,2$.
Since $|\pi^{-1}(q_i')|=1$,
we get (iii)(c).

If 
$\pi_1^{-1}(q_1)=\{n_{j_1}^{(1)},n_{j_2}^{(1)}\}$
and
$\pi_2^{-1}(q_2)=\{n_{j_2}^{(2)},n_{j_3}^{(2)}\}$
for some distinct ${j_1},{j_2}, j_3\in\{1,\ldots,\delta\}$, 
then in particular
$e_{\pi_1}(n_{j_1}^{(1)})=e_{\pi_1}(n_{j_2}^{(1)})=1$
and
$e_{\pi_2}(n_{j_2}^{(2)})=e_{\pi_2}(n_{j_3}^{(2)})=1$.
Note that 
${\pi_1}(n_{j_3}^{(1)})=q_1'$
and
${\pi_2}(n_{j_1}^{(2)})=q_2'$
and we again get (iii)(c). Note that 
$e_{\pi_1}(n_{j_3}^{(1)})$ and $e_{\pi_2}(n_{j_1}^{(2)})$ can be either 1 or 2.

Lastly, we note that
$\pi_1^{-1}(q_1)=\{n_{j_1}^{(1)},n_{j_2}^{(1)}\}$
and
$\pi_2^{-1}(q_2)=\{n_{j_3}^{(2)},n_{j_4}^{(2)}\}$,
with ${j_1},{j_2}, j_3,j_4\in\{1,\ldots,\delta\}$ distinct,
does not happen, since $\delta\leq 3$. 
\end{proof}

\bibliographystyle{aalpha}

\bigskip
\noindent{\smallsc Juliana Coelho, Universidade Federal Fluminense (UFF), 
Instituto de Matem\'atica e Estat\'istica - 
 Rua Prof. Marcos Waldemar de Freitas Reis, S/N, 
 Gragoat\'a, 24210-201 - Niter\'oi -  RJ,  Brasil}\\
{\smallsl E-mail address: \small\verb?julianacoelhochaves@id.uff.br?}
\bigskip
\bigskip

\noindent{\smallsc Frederico Sercio, Universidade Federal de Juiz de Fora  (UFJF),  Departamento de Matem\'{a}tica, Rua Jos\'{e} Louren\c{c}o Kelmer, s/n - Campus Universit\'{a}rio, 36036-900 - Juiz de Fora - MG, Brasil}\\
{\smallsl E-mail address: \small\verb?fred.feitosa@ufjf.edu.br?}

\end{document}